\documentclass[11.5pt]{article}
\usepackage{graphicx}  
\usepackage{colortbl}
\usepackage[most]{tcolorbox}
\usepackage{fourier}
\usepackage{faktor}
\makeatletter
\DeclareRobustCommand*{\mfaktor}[3][]
{
   { \mathpalette{\mfaktor@impl@}{{#1}{#2}{#3}} }
}
\newcommand*{\mfaktor@impl@}[2]{\mfaktor@impl#1#2}
\newcommand*{\mfaktor@impl}[4]{
   \settoheight{\faktor@zaehlerhoehe}{\ensuremath{#1#2{#3}}} 
   \settoheight{\faktor@nennerhoehe}{\ensuremath{#1#2{#4}}} 
      \raisebox{-0.5\faktor@zaehlerhoehe}{\ensuremath{#1#2{#3}}} 
      \mkern-4mu\diagdown\mkern-5mu 
      \raisebox{0.5\faktor@nennerhoehe}{\ensuremath{#1#2{#4}}} 
}
\makeatother

\usepackage[left=3.5cm,right=3.5cm,bottom=3.5cm,top=3.5cm]{geometry}
\usepackage{pstricks}
\usepackage{arydshln}
\usepackage{cancel}
\usepackage{subfig}
\usepackage{xcolor}
\usepackage[utf8]{inputenc}
\usepackage{tikz}
\usepackage{color}
\usepackage{braket}
\usepackage{amstext}
\usepackage{amsthm}
\usepackage{graphicx}
\usepackage{soul}
\usepackage[spanish,english]{babel}
\usepackage{physics}
\definecolor{Fg}{RGB}{10,100,85}
\definecolor{Gr}{RGB}{220,240,245}
\usepackage{amsmath}
\usepackage{url}
\usepackage{graphicx}
\usepackage{amssymb}
\usepackage{float}
\usepackage{lipsum}
\def\acts{\curvearrowright}

\newcommand{\N}{\mathbb{N}}

\newcommand{\C}{\mathbb{C}}
\newcommand{\Z}{\mathbb{Z}}
\newcommand{\R}{\mathbb{R}}
\newcommand{\Pp}{\mathbb{P}}
\newcommand{\Hh}{\mathbb{H}}
\newcommand{\K}{\mathbb{K}}

\newcommand{\id}{\mathbb{1}}

\newcommand{\slr}{\mathfrak{sl}_4\R}
\newcommand{\Slr}{\mathrm{SL_4\R}}
\newcommand{\Pslr}{\mathrm{PSL_4\R}}
\newcommand{\Stab}{\mathrm{Stab}}
\newcommand{\dev}{\mathrm{dev}}
\newcommand{\Ad}{\mathrm{Ad}}

\newcommand{\Arh}{\mathrm{Ad}\rhp}

\newcommand{\hol}{\mathrm{hol}}
\newcommand{\Rep}{\mathrm{Rep}}

\newtheorem{thm}{Theorem}[section]
\newtheorem{lmm}[thm]{Lemma}
\newtheorem{prp}[thm]{Proposition}
\newtheorem{cor}[thm]{Corollary}
\newtheorem{de}[thm]{Definition}

\newtheorem{ex}[thm]{Example}
\newtheorem{rmk}[thm]{Remark}

\newcommand{\Orb}{\mathcal{O}}
\newcommand{\rhp}{\rho_{\mathrm{hyp}}}
\newcommand{\chp}{\chi_{\mathrm{hyp}}}

\author{Alejandro García and Joan Porti\footnote{Both authors partially supported by
FEDER-AEI (grant numbers PID2021-125625NB-100 and
Mar\'\i a de Maeztu Program CEX2020-001084-M) The second author is supported by the doctoral grant PRE2022-105108}}
\date\today
\title{Projective deformations of hyperbolic 3-orbifolds with turnover ends}
\begin{document}
\maketitle
\begin{abstract}
We study projective deformations of (topologically finite) hyperbolic 3-orbifolds whose ends have turnover cross section. These deformations are examples of \textit{projective cusp openings}, meaning that hyperbolic cusps are deformed in the projective setting such that they become totally geodesic generalized cusps with diagonal holonomy. We find that this kind of structure is the only one that can arise when deforming hyperbolic turnover cusps, and that turnover funnels remain totally geodesic. Therefore, we argue that,  under no infinitesimal rigidity assumptions, the deformed projective 3-orbifold remains properly convex. Additionally, we give a complete description of the character variety $X(\pi_1S^2(3,3,3),\Slr)$. 
\end{abstract}
\section*{Introduction}\label{0}
\setcounter{section}{0}
Hyperbolic structures on finite volume 3-orbifolds are known to be rigid, both infinitesimally (Weil rigidity) and globally (Mostow-Prasad rigidity). Then, the hierarchy of geometries given by the inclusion $\mathrm{Aut}(\Hh^3)<\mathrm{Aut}(\Pp^3)$ yields the natural question of whether finite volume hyperbolic 3-orbifolds are projectively rigid. There is not a generic answer, since a variety of flexible and rigid examples is known (see, e.g., the work of Cooper, Long, and Thistlethwaite in \cite{MR2264468} for closed 3-manifolds). 
\par In the setting of flexible noncompact manifolds or orbifolds, an interesting phenomenon arises: when deforming a hyperbolic cusp (whose holonomy preserves a foliation by horospheres), the deformed cusp sometimes becomes \textit{totally geodesic} (so the deformed holonomy preserves a projective plane). Note that by deformation we always mean \textit{small/local deformation} (for the global notion, we use \textit{path of deformations}). Ballas, Danciger, and Lee obtained in \cite{MR3780442} that, provided some infinitesimal rigidity premise, any oriented finite volume hyperbolic 3-manifold admits a properly convex projective structure close to the hyperbolic one where each cusp becomes totally geodesic. 
\par In this article, we examine a class of hyperbolic 3-orbifolds. Namely, we require any end to have turnover cross section. Recall that a turnover is the double of a triangle $T(n_1,n_2,n_3)\subset P_k$ of angles $\frac{2\pi}{n_1},\frac{2\pi}{n_2},\frac{2\pi}{n_3}$ lying in a plane of constant curvature $k\in\{\pm1,0\}$. We follow the notation $S^2(n_1,n_2,n_3)$ for the double of $T(n_1,n_2,n_3)$ with the aim of remarking that it is diffeomorphic to a sphere with three singular points of order $n_1,n_2,n_3$. Turnover ends are products $[0,\infty)\times S^2(n_1,n_2,n_3)$ where $\sum_i n_i=1$ (so the turnover is Euclidean and the end is a cusp) or $\sum_i n_i<1$ (so the turnover is hyperbolic and the end is a funnel).
\par The class of hyperbolic orbifolds whose ends are of turnover type is interesting even if we drop the finite-volume condition, since hyperbolic turnovers are rigid in the hyperbolic space $\Hh^3$. Thus, this class consists of hyperbolically rigid orbifolds, and we consider the problem of projective deformations of them. Under no infinitesimal rigidity assumptions, our principal result is:
\begin{thm}\label{establ}
    Let $\Orb^3$ be a topologically finite hyperbolic 3-orbifold whose ends have turnover cross section. Then, any path of projective deformations of $\Orb^3$ produces a properly convex projective 3-orbifold whose ends are either hyperbolic cusps or totally geodesic.
\end{thm}
\par Notice that this theorem has a global flavor and, in particular, implies a local version (closer to the one in \cite{MR3780442} we mentioned above). Moreover, it provides examples of realization of the claim by Ballas-Cooper-Leitner in \cite{MR4125754} that \textit{generalized cusps appear by projective deformations of hyperbolic cusp}, and also a global converse of such a statement for orbifolds as in Theorem \ref{establ}.
\par The proof of Theorem \ref{establ} is more delicate for the cusped case, and is based on the study of representations of $\pi_1\left(S^2(3,3,3)\right)$ in $\Slr$, because the unique Euclidean turnover that is not projectively rigid is $S^2(3,3,3)$ (see Lemma \ref{ultimolem}). We distinguish two principal ingredients from that proof:
\par First, we adapt a partial slice in $\Rep(\pi_1T^2,\Slr)$ around the hyperbolic holonomy $\rhp$, built by Ballas-Danciger-Lee in \cite{MR3780442} to study projective deformations of the torus cusp. This partial slice becomes a slice with no restrictions in our case.
\begin{thm}\label{main1}
There is a slice $\mathcal{S}_{\rhp}\subset\Rep\left(\pi_1\left(S^2(3,3,3)\right),\Slr\right)$ at $\rhp$ that represents all (conjugacy classes of) projective holonomies near $\rhp$. Moreover, the restriction to $\pi_1(T^2)$ of such holonomies is either diagonalizable or induced by a hyperbolic cusp.
\end{thm}
\par Secondly, we obtain a global description of the character variety $X(\pi_1\left(S^2(3,3,3)\right),\Slr)$. Recall that real character varieties are semialgebraic subsets (as ensured by real GIT; see \cite{MR1087217}). We describe $X(\pi_1\left(S^2(3,3,3)\right),\Slr)$ both algebraically and topologically as follows.
\begin{thm}\label{mzcla}
        The character variety $X(\pi_1\left(S^2(3,3,3)\right),\Slr)$ is not only real semialgebraic but a real algebraic set. It has a unique irreducible component $X_0\subset X(\pi_1\left(S^2(3,3,3)\right),\Slr)$ of positive dimension, which is a hypersurface in $\R^3$ given implicitly by
        \begin{equation*}
              \tau^2-(rs-3)\tau+r^3+s^3-6rs+9=0,
        \end{equation*}
        where the coordinates $r,s,\tau$ are traces of matrices in $\rho\left(\pi_1(S^2(3,3,3)\right)$.
       \par Topologically, $X_0=S_1\sqcup S_2$ splits into two disjoint components
     \begin{align*}
         &S_1=X_0\cap \{r\geq 3,s\geq 3\},\\
         &S_2=X_0\setminus S_1
     \end{align*}
     homeomorphic to $\R^2$. Moreover, the projection $$\pi:\Rep(\pi_1\left(S^2(3,3,3)\right),\Slr)\to X(\pi_1\left(S^2(3,3,3)\right),\Slr)$$ restricted to the slice $\mathcal{S}_{\rhp}$ is a surjection onto $S_1$ which is 3-to-1 generically and 1-to-1 at the hyperbolic holonomy $\rhp$. Thus, $S_1$ has a natural orbifold structure with a sole branching point $\chp$ of order three. Further, $S_2$ is smooth.
    \par On the other hand, the complement of $X_0$ consists of $19$ isolated points with no geometric realization, i.e. representations at their fibers are not faithful.
\end{thm}
\paragraph{Organization of this paper} We devote Section \ref{1} to state briefly some preliminaries we need. In Section \ref{2}, we concentrate on algebraic deformations of the end $[0,\infty)\times S^2(3,3,3)$, which is the only flexible turnover cusp. Theorem \ref{main1} (which will appear as Theorem \ref{main}) is stated and proved throughout that section. We apply this algebraic understanding to the geometric setting in
Section \ref{3}; Theorem \ref{thm}, a local version of Theorem \ref{establ}, is the main result there. Finally, in Section \ref{4}, we face to the generic problem of projective structures on $[0,\infty)\times S^2(3,3,3)$. Then, we get Theorem \ref{mzcla} (see Theorem \ref{charvar} and Remark \ref{generalization}). As a consequence, we are able to prove our main result (Theorem \ref{establ}), which appears as Corollary \ref{generalization}. 
\section{Preliminaries}\label{1}
In this section, we recall general notions that we shall use extensively. Everything along these preliminaries is standard, so there are many good references explaining the details. We mention some of them below. For geometric structures, we refer to Goldman's book \cite{MR4500072}, since there are no relevant differences between the theory for manifolds and good orbifolds. The orbifold generalization was done by Choi in \cite{MR2043960}. In addition, we refer to the paper by Sikora \cite{MR2931326} for the theory of $\mathrm{SL}_n\C$ representation and character varieties. 
\par All the orbifolds we deal with along this article are good orbifolds, i.e. they are quotients of manifolds by discrete groups. Even if no specific preliminaries about orbifolds are needed, the definition and classification of turnovers is assumed. From the geometric side, we are interested in $(G,X)$-structures, which locally model such orbifolds in some homogeneous space $X$ with transition functions in $G$ (a topological group acting on $X$ by diffeomorphisms). By gluing the charts, we get a \textbf{developing map} $\dev:\tilde{\Orb}\to X$ from the universal cover to the homogeneous space. Given a $(G,X)$-structure on $\Orb$ with developing map $\dev$, a representation $\rho:\pi_1^{orb}(\Orb)\to G$ of the orbifold fundamental group in $G$ is called the \textbf{holonomy} if $\dev$ is $\rho$-equivariant. We say that $(\dev,\rho)$ is a \textbf{developing pair}, and it can be proved that such pairs are equivalent to the data of a geometric structure for $\Orb$. 
\par We focus on hyperbolic and projective geometries on 3-orbifolds, i.e. $(\mathrm{SO(3,1)},\mathbb{H}^3)$ and $(\mathrm{PSL_4\R},\R \mathbb{P}^3)$ structures, respectively. Since there is a projective model of the hyperbolic space $\Hh ^n$, any hyperbolic orbifold can be realized as a projective one. Moreover, hyperbolic orbifolds are included in a subclass of projective orbifolds, namely they have a \textbf{properly convex projective structure}. Such a subclass consists of quotients $\Omega/\Gamma$, where $\Omega\subset \R^n\subset \R\mathbb{P}^n$ is a convex bounded domain in some affine chart, and $\Gamma<\mathrm{PSL_{n+1}\R}$ is a discrete subgroup. In terms of the developing pair, this means that the developing map has convex image
\begin{equation*}
    \dev(\Orb)\cong\Omega
\end{equation*}
and $$\rho\left(\pi_1^{orb}(\Orb)\right)=\Gamma<G$$
is discrete.
\par The map assigning to each structure its holonomy representation modulo conjugation in $G$ is called the \textbf{holonomy map}, and denoted by $\hol$. If $\Orb$ is a compact orbifold, then $\hol$ is locally surjective, so we can locally deform our structure by deforming its holonomy. In the noncompact case, this is not so clear. However, a similar holonomy principle due to Cooper, Long and Tillmann \cite{MR3336086} (see Ballas-Danciger-Lee \cite{MR3780442} for a precise statement) holds for deformations of generalized cusps.
\par Now, we state some general facts about representation varieties. In what follows, we assume $G$ is a complex (or real) semisimple algebraic Lie group. Consider the set of homomorphisms $\mathrm{Hom}(\pi_1\Orb,G)$, which contains the holonomy representation of any $(G,X)$-structure on $\Orb$. Topologize $\mathrm{Hom}(\pi_1\Orb,G)$ with the subspace topology of the set of continuous functions (with respect to the compact-open topology) $\pi_1\Orb\to G$; the resulting topological space is the \textbf{representation variety} $\Rep(\pi_1\Orb,G)$. If $G$ is complex and $\pi_1\Orb$ is finitely generated, $\Rep(\pi_1\Orb,G)$ is an algebraic subset of $G^n$, where $n$ is the cardinal of some generating set of $\pi_1\Orb$. Moreover, if $G_{\R}$ is the subgroup of real points of a complex algebraic Lie group $G$, $\Rep(\pi_1\Orb,G_{\R})$ is the subset of real points $\Rep(\pi_1\Orb,G_\C)$, and it is an algebraic subset of $G_{\R}^n$.
\par We said that $\hol$ sends a geometric structure to its holonomy up to conjugation. Thus, the codomain of $\hol$ is the \textbf{character variety}
\begin{equation*}
    X(\pi_1\Orb,G):=\Rep(\pi_1\Orb,G)/\!/G,
\end{equation*}
where $/\!/$ means the GIT quotient if $G$ is complex or the real GIT quotient as developed in Richardson-Slodowy \cite{MR1087217} (see also Böhm-Lafuente \cite{MR4420784}) if $G$ is real. Topologically, $X(\pi_1\Orb,G)$ is the space of the closure of the orbits by conjugation rather than simply the space of such orbits. Richardson-Slodowy \cite{MR1087217} showed that the closure of any orbit contains exactly a closed orbit corresponding to a semisimple representation, so there is an equivalent description of the character variety given by
$$X(\pi_1\Orb,G)=\Rep^{SS}(\pi_1\Orb,G)/G,$$
where $\Rep^{SS}(\pi_1\Orb,G)$ stands for semisimple $\pi_1\Orb\to G$ representations. As a consequence of (real or complex) GIT, the character variety is a Hausdorff space (however, $\Rep(\pi_1\Orb,G)/G$ is not). Moreover, the crucial theorem in \cite{MR1087217} states that $X(\pi_1\Orb,G)$ is a semialgebraic set if $G$ is real (recall that $X(\pi_1\Orb,G)$ is an algebraic set if $G$ is complex).
\section{Projective deformations of $[0,\infty)\times S^2(3,3,3)$}\label{2}
 Consider a hyperbolic cusp $[0,\infty)\times S^2(3,3,3)$. Along the whole paper, let $$\Gamma:=\pi_1^{orb}(S^2(3,3,3))=<a,b|a^3=b^3=(ab)^3=1>$$ be the orbifold fundamental group of $S^2(3,3,3)$. Then there is a lattice subgroup $\Gamma_0<\Gamma$ given by
$$\Gamma_0:=<a^2b,ba^2>\cong\Z^2$$
which fits in the following exact sequence
\begin{equation}\label{exsq}
    1\rightarrow\Gamma_0\rightarrow\Gamma\rightarrow \Z_3\rightarrow 1. 
\end{equation}
\begin{figure}[ht]
	\begin{center}
		\begin{tikzpicture}[line join = round, line cap = round, scale=.8]
			\draw[thick, color=blue]  (-1,0)--(1,0)--(0, 1.73)--cycle;
			\draw[red, fill=red] (-1,0) circle(0.04);
			\draw[red, fill=red] (1,0) circle(0.04);
			\draw[red, fill=red] (0,1.73) circle(0.04);
 			 \draw[very thin] (-.6,.7) arc[x radius=.6, y radius=.15, start angle = 180, end angle= 360] ;
 			 \draw[very thin, dashed] (-.6,.7) arc[x radius=.6, y radius=.15, start angle = 180, end angle= 0] ;
			\draw[very thin] (-1+.4,0) to[out=80, in=0] (-1+.4*.5  , 0+.4*.87);
			\draw[dashed, very thin] (-1+.4,0) to[out=160, in=-100] (-1+.4*.5  , 0+.4*.87);
			\draw[very thin] (1-.4,0) to[out=110, in=180+0] (1-.4*.5  , 0+.4*.87);
			\draw[ dashed, very thin] (1-.4,0) to[out=20, in=-80] (1-.4*.5  , 0+.4*.87);
			\draw[very thin]  (0-.4*.5, 1.73-.4*.87) to[out=-60, in=180+60] (0+.4*.5, 1.73-.4*.87);
			\draw[ dashed, very thin]  (0-.4*.5, 1.73-.4*.87) to[out=30, in=180-30] (0+.4*.5, 1.73-.4*.87);
			\draw (-6/5,0) node{$3 $};
			\draw (6/5,0) node{$3 $};
			\draw (1/10,2) node{$3 $};
	\end{tikzpicture}
\end{center}
\caption{The turnover $S^2(3,3,3)$.}
\label{Figure:Turnover}
\end{figure}
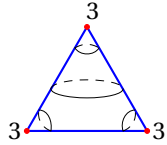
\par Let us denote by $\rhp$ the holonomy of the hyperbolic structure of $[0,\infty)\times S^2(3,3,3)$. By \cite{MR3336086}, we know that $\rhp|_{\Gamma_0}$ can be regarded as a generalized cusp of the $C_0$-type (following the nomenclature of that paper). Namely, we can take
\begin{equation}\label{g1}
\rhp(a^2b),\rhp(ba^2)\in C_0:=\left\{M(x,y):=\begin{bmatrix}
1 & x & y & \frac{x^2+y^2}{2} \\
 & 1 & 0 & x \\
 &  & 1 & y \\
 &  &  & 1 
\end{bmatrix}:(x,y)\in \R^2\right\}<\Slr.
\end{equation}
Notice that $C_0$ is isomorphic to $\R^2$ as a Lie group. Further, for any $re^{i\theta}\in\C^*$, the similarity 
\begin{equation*}
\Omega_{(r,\theta)}:=\begin{bmatrix}
r &  &  &  \\
 & \cos(\theta) & -\sin(\theta)&  \\
 & \sin(\theta) & \cos(\theta) &  \\
 &  &  & \frac{1}{r} 
\end{bmatrix}\in \Slr 
\end{equation*}
acts by conjugation on $C_0$ via 
$$\Ad{\Omega_{(r,\theta)}}\left( M(x,y)\right)=M(\varphi_{(r,\theta)}(x,y)),$$
where 
$\varphi_{(r,\theta)}:\R^2\xrightarrow{}\R^2$ is the composition $\varphi_{(r,\theta)}=D_r\circ R_\theta$ of a rotation of angle $\theta$ and a dilation of factor $r$. (Or equivalently, a complex dilation of factor $re^{i\theta}$). Hence, the group of similarities acts transitively on $C_0^*:=C_0\setminus\{\id\}$. On the other hand, $a$ is an element of order three acting by conjugation on $C_0$, so $\rhp(a)$ must be a $2\pi/3$ rotation regarded as automorphism of the $(x,y)$-plane. Thus,
\begin{equation}\label{a}
    \rhp(a)=\Omega_{\left(1,\frac{2\pi}{3}\right)}.
\end{equation}
By the condition in Equation \eqref{g1}, we can take $(x_0,y_0)\neq(0,0)$ to define
\begin{equation}\label{x0y0}
    \rhp(a^2b)=M(x_0,y_0).
\end{equation}
Then,
\begin{equation}
    \rhp(ba^2)= \rhp(a(a^2b)a^2)=M\left(R_{\frac{2\pi}{3}}(x_0,y_0)\right).
\end{equation} 
The transitivity of the action of the group of similarities on $C_0^*$ means that any choice of $(x_0,y_0)\in\left(\R^2\right)^*$ in Equation (\ref{x0y0}) is geometrically equivalent. From now on, we set $(x_0,y_0)=(1,0)$, namely:
   \begin{equation}
       \rhp(a^2b)=M(1,0), \quad \rhp(ba^2)=M\left(\cos(\frac{2\pi}{3}),\sin(\frac{2\pi}{3})\right), \quad \rhp(a)=\Omega_{\left(1,\frac{2\pi}{3}\right)}.
   \end{equation}
Now, let us start to analyze the variety of representations  $\Rep(\Gamma,\slr)$ near $\rhp$. For this purpose, the following group cohomology results are useful.
\begin{prp}\label{0coc}
    The $0$-cohomology groups satisfy:
        \begin{equation}
\dim\left(H^0(<a^2b>,\Arh)\right)=\dim\left(H^0(<ba^2>,\slr)\right)=5.
    \end{equation}
    \begin{equation}
        \dim\left(H^0(\Gamma_0,\Arh)\right)=3,
    \end{equation}
    and
    \begin{equation}\label{eq10coh}
        \dim\left(H^0(\Gamma,\Arh)\right)=1,
    \end{equation}
\end{prp}
\begin{proof}
   Let us describe the cohomology groups via invariant subspaces of $\slr$:
    \begin{align*}
        &H^0(<a^2b>,\Arh)\cong\left(\slr\right)^{\rhp(<a^2b>)},\\ &H^0(\Gamma_0,\Arh)\cong\left(\slr\right)^{\rhp(\Gamma_0)}=\left(\left(\slr\right)^{\rhp(<a^2b>)}\right)^{\rhp(<ba^2>)},\\        &H^0(\Gamma,\Arh)\cong\left(\slr\right)^{\rhp(\Gamma)}=\left(\left(\slr\right)^{\rhp(\Gamma_0)}\right)^{\rhp(<a>)}.
    \end{align*}
   Elementary computations yield
    \begin{align*}
        &\left(\slr\right)^{\rhp(<a^2b>)}\cong\left\{ \begin{bmatrix}
0 & x_1 & x_2 & x_3 \\
 &  &  & x_4 \\
 &  &  & x_5 \\
 &  &  &  0
\end{bmatrix}: x_1,x_2,x_3,x_4,x_5\in \R\right\}.
    \end{align*}
Moreover, $\Ad{\Omega_{\left(1,\frac{-2\pi}{3}\right)}}\left(\rhp(a^2b)\right)=\rhp(ba^2)$ induces a natural isomorphism $$\left(\slr\right)^{\rhp(<a^2b>)}\cong\left(\slr\right)^{\rhp(<ba^2>)}.$$
Thus, the intersection 
$$\left(\slr\right)^{\rhp(<a^2b>)}\cap \mathrm{Ad}\Omega_{\left(1,\frac{2\pi}{3}\right)}\left( \left(\slr\right)^{\rhp(<a^2b>)}\right)=\left\{ \begin{bmatrix}
 0& x_1 & x_2 & x_3 \\
 &  &  & x_1 \\
 &  &  & x_2 \\
 &  &  &  0
\end{bmatrix}: x_1,x_2,x_3\in \R\right\}$$
is isomorphic to $\left(\slr\right)^{\rhp(\Gamma_0)}$. Finally, 
$$\left(\slr\right)^{\rhp(\Gamma)}\cong \left(\slr\right)^{\rhp(\Gamma_0)}\cap \Ad\Omega_{\left(1,\frac{2\pi}{3}\right)}\left(\left(\slr\right)^{\rhp(\Gamma_0)}\right)=\left\{ \begin{bmatrix}
 0&0  &0  & x \\
 &  &  & 0 \\
 &  &  &  0\\
 &  &  &  0
\end{bmatrix}: x\in \R\right\}.$$
\end{proof}
Since $\rhp(\Gamma)$ fixes a line in $\R^4$, it is not irreducible. However, it has a nice property we call being strongly regular (from \cite{MR4495671}). Actually, \cite{MR4495671} deals with a complex Lie group, but we can ignore that since the inclusion $\Rep(\Gamma,\Slr)\subset\Rep(\Gamma,\mathrm{SL_4\C})$ is smooth.
\begin{de}
    A Euclidean turnover $\Orb^2$ group representation $\rho:\pi_1(\Orb^2)\to G$ into a Lie group $G$ is \textbf{strongly regular} if, for the torus fundamental group $\Z\oplus\Z<\pi_1(O^2)$, then $\dim \mathfrak{g}^{\rho(\Z\oplus\Z)}=\rank G$ and
    the projection of $\rho(\Z\oplus\Z)$ on $G/\mathrm{Z}(G)$, where $\mathrm{Z}(G)$ is the center of $G$, is contained in a connected abelian subgroup.
\end{de}
\begin{lmm}
   The representation $\rhp$ is strongly regular in $\Rep(\Gamma,\Slr)$.
\end{lmm}
\begin{proof}
    Recall that the Cartan subalgebra $\mathfrak{C}(\slr)$ of $\slr$ is the one of traceless matrices, hence
    $$\rank(\Slr)=\dim\mathfrak{C}(\slr))=3.$$
    Thus, from Proposition \ref{0coc},
    $$\dim\left(H^0(\Gamma_0,\Arh)\right)=\rank(\Slr).$$
Also, let $Z=\{\pm \id\}$ be the center of $\Slr$ and $\pi:\Slr\to \Slr/Z=\Pslr$ the natural projection. Since $\rhp(\Gamma_0)$ si a subset of $C_0$, which is linearly homeomorphic to $\R^2$, $\pi\left(\rhp(\Gamma_0)\right)$ is contained in the abelian connected subgroup $\pi(C_0)$.
\end{proof}
Before we state and prove the following lemma, let us recall the notion of group cohomology in degree 1. We define $H^1(\Gamma,\Arh)$ to be the quotient of $Z^1(\Gamma,\Arh)/B^1(\Gamma,\Arh)$, where $Z^1(\Gamma,\Arh)$ is the group of twisted morphisms $\Gamma\to\slr$ with respect to the action $\Arh$, and $B^1(\Gamma,\Arh)$ is the normal subgroup of inner cocycles. For a more detailed explanation, see Section 1.4 in Löh \cite{loh}.
\begin{lmm}\label{lema}
    The representation $\rhp$ is a smooth point of $\Rep(\Gamma,\Slr)$.  Thus,
    \begin{equation}\label{eqlemaTan}
        \dim \Rep_0(\Gamma,\Slr)=16
    \end{equation}
    and
    \begin{equation}
       \dim X_0(\Gamma,\Slr)= \dim H^1(\Gamma,\Ad \rhp),\label{eqn}
    \end{equation}
    where $\Rep_0(\Gamma,\Slr)$ is the component of $\Rep(\Gamma,\Slr)$ containing $\rhp$ and $\mathrm{X}_0(\Gamma,\Slr)$ is the component of $\mathrm{X}(\Gamma,\Slr)$ containing the character of $\rhp$. 
\end{lmm}
\begin{proof}
    To be short, let $\Orb^2:=S^2(3,3,3)$. Smoothness is a straight consequence of Theorem 4.11 in \cite{MR4495671}, and so are the formulae (\ref{eqn}) and
    $$\dim \Rep(\Gamma,\Slr)=\Tilde{\chi}(\Orb^2,\Ad\rhp)+\dim \Slr+ \dim\left(\slr\right)^{\Arh(\Gamma)},$$
    where $\Tilde{\chi}(\Orb^2,\Ad\rhp)$ is the \textit{twisted Euler characteristic} (Definition 3.1 in \cite{MR4495671}). Let us denote by $\Sigma$ the set of branching points of $\Orb^2$. Then, $\Orb^2\setminus\Sigma$ is topologically a thrice-punctured sphere. Hence,
    \begin{equation}\label{todoturn}
        \tilde{\chi}(\Orb^2,\Ad\rhp)=\chi(\Orb^2\setminus\Sigma)\dim \Slr+\sum_{v\in \Sigma} \dim\left(H^0(\Stab(v),\Arh)\right)=-15+3\cdot 5=0.
    \end{equation}
    Thus,  $$\dim \Rep(\Gamma,\Slr)=15+1=16.$$
\end{proof}
\begin{rmk}
    As an immediate corollary of smoothness, we have the natural identification
    \begin{equation*}
        T_{\rhp}\Rep(\Gamma,\Slr)\cong Z^1(\Gamma,\Ad \rhp).
    \end{equation*}
\end{rmk}
\begin{prp}\label{tangent}
    $\dim \mathrm{X}_0(\Gamma,\Slr)=2$.\label{dimslice}
    \end{prp}
\begin{proof}
    By equality (\ref{eqn}) in Lemma \ref{lema}, it suffices to calculate $\dim H^1(\Gamma,\Arh)$. Moreover, we know from the proof of Lemma \ref{lema} that the twisted Euler characteristic at the hyperbolic holonomy is $ \tilde{\chi}(\Orb^2,\Ad\rhp)=0$. Hence,
    \begin{equation*}
        \dim H^0(\Gamma,\Arh)-\dim H^1(\Gamma,\Arh)+\dim H^2(\Gamma,\Arh)=0.
    \end{equation*}
    By Poincaré duality, $\dim H^0(\Gamma,\Arh)=\dim H^2(\Gamma,\Arh)$, so
    $$\dim H^1(\Gamma,\Arh)=2\dim H^0(\Gamma,\Arh)=2,$$ 
    where we have used that $\dim H^0(\Gamma,\Arh)=1$ (see Proposition \ref{0coc}).
\end{proof}
Now, to state and prove the main result (Theorem \ref{main}) of this section, we give a precise definition of what we mean by `slice'.
\begin{de}[Slice]\label{slice}
      A submanifold $\mathcal{S}_{\rho}<\Rep(\Gamma,\Slr)$ containing a smooth point $\rho\in \Rep(\Gamma,\Slr)$ is a \textbf{slice at $\rho$} if the natural isomorphism $T_\rho \Rep(\Gamma,\Slr)\cong T_{\rho} \mathcal{S}_{\rho} \oplus B^1(\Gamma,\slr)$ holds.
\end{de}
\begin{thm}\label{main}
There is a slice $\mathcal{S}_{\rhp}<\Rep(\Gamma,\Slr)$ at $\rhp$ which represents all the (conjugacy classes of) projective holonomies near $\rhp$. Moreover, the restriction to $\Gamma_0$ of such holonomies is either diagonalizable or hyperbolic.
\end{thm}
\begin{proof}
    By Proposition \ref{dimslice}, a slice at $\rhp$ needs to have dimension $2$. Let $\mathcal{S}_{\rhp}:=\Psi_{(\R^2)}$ be a 2-dimensional submanifold of $\Rep(\Gamma,\Slr)$ given by 
    \begin{equation}
         \Psi_{(u,v)}(a^2b)=\exp \begin{bmatrix}
 0& 1 & 0 &0  \\
0 & u & v & 1  \\
0 & v & -u & 0\\
 0& 2(u^2+v^2) &  0&0  
\end{bmatrix},\quad\Psi_{(u,v)}(a)=\rhp(a),
    \end{equation}
for any $(u,v)\in\R^2$. Since $\Psi_{(0,0)}(a^2b)=\rhp(a^2b)$ and $<a^2b,a>=\Gamma$, it is clear that $\Psi_{(0,0)}=\rhp$. 
\par We refer to \cite{MR3780442} to see that $\Psi_{(u,v)}(a^2b)$ is diagonalizable whenever $(u,v)\neq (0,0)$. For $(u,v)=t(\cos{3s},\sin{3s})\in \left(\R^2\right)^*$, the authors also parametrize the eigenvectors of $\Psi_{(u,v)}(a^2b)$ (regarded as projective points in $\R\Pp^3$ or vectors in $\R^4$, as convenient) by
\begin{align}
    & p_\infty=(1,0,0,0)
    ,\label{param1}\\
    & p_1(t,\theta)=\Omega_{(1,\theta)}(1,2t,0,2t^2)^t 
    ,\label{param2}\\
    & p_2(t,\theta)=\Omega_{\left(1,\frac{2\pi}{3}\right)}p_1(t,\theta)
    ,\label{param3}\\
    & p_3(t,\theta)=\Omega_{\left(1,\frac{2\pi}{3}\right)}p_2(t,\theta)
    \label{param4}
    ,
\end{align}
whose associated eigenvalues satisfy exactly the relations $\lambda_\infty=1=\lambda_1\lambda_2\lambda_3,$ and $\lambda_1,\lambda_2,\lambda_3>0$. 
\par To see that $\Psi_{(u,v)}$ is actually a homomorphism for $(u,v)\in \left(\R^2\right)^*$, we need to check that 
\begin{equation}\label{commute}
     \left(\left(\Psi_{(u,v)}(a)\right)^2\Psi_{(u,v)}(a^2b)\right)^3=\id, 
 \end{equation}
and
\begin{equation}\label{25}
\left(\Psi_{(u,v)}(a)\Psi_{(u,v)}(a^2b)\right)^3=\id.
\end{equation}
The parametrization above (Equations (\ref{param1} - \ref{param4})) guides us to think of $\Psi_{(u,v)}(a)$ as a projective permutation of the vertices of the triangle $\triangle\left(p_1(t,\theta), p_2(t,\theta), p_3(t,\theta)\right)$ fixing $p_\infty$. Indeed, the conjugation
\begin{equation*}
[p_\infty|p_1(t,\theta)|p_2(t,\theta)|p_3(t,\theta)]^{-1}\Psi_{(u,v)}(a)[p_\infty|p_1(t,\theta)|p_2(t,\theta)|p_3(t,\theta)]
\end{equation*}
is the permutation of order three
\begin{equation*}
          \begin{bmatrix}
 1& 0&  0& 0\\
0&  &  & 1 \\
 0& 1 &  & \\
 0&  & 1&  
\end{bmatrix}.
    \end{equation*}
 Moreover, it implies that $\Psi_{(u,v)}(a)\Psi_{(u,v)}(a^2b)\Psi_{(u,v)}(a^2)$ shares a common basis of eigenvectors with $\Psi_{(u,v)}(a^2b)$. Hence $\Psi_{(u,v)}(a^2b)$ and $\Psi_{(u,v)}(a)\Psi_{(u,v)}(a^2b)\Psi_{(u,v)}(a^2)$ commute, which implies that relation (\ref{commute}) stands. Now, to check that Equation (\ref{25}), let us conjugate $\Psi_{(u,v)}$ so that $\Psi_{(u,v)}(a^2b)$ is diagonal. Then, we obtain:
 \begin{equation}
 \Psi_{(u,v)}(a)\Psi_{(u,v)}(a^2b)=\begin{bmatrix}
 1& 0&  0& 0\\
0&  &  & 1 \\
 0& 1 &  & \\
 0&  & 1&  
\end{bmatrix}   \begin{bmatrix}
 1& &  & \\
&  \lambda_1&  &  \\
 &  &\lambda_2  & \\
 &  & &\lambda_3  
\end{bmatrix}=
\begin{bmatrix}
 1& 0&  0& 0\\
0&  &  & \lambda_3 \\
 0& \lambda_1 &  & \\
 0&  & \lambda_2&  
\end{bmatrix}.  
 \end{equation}
Therefore, $$\left(\Psi_{(u,v)}(a)\Psi_{(u,v)}(a^2b)\right)^3=\lambda_1\lambda_2\lambda_3\id.$$
Thus, relation $\lambda_1\lambda_2\lambda_3=1$ implies that Equation (\ref{25}) holds.
 \par Now, we shall show that $\Psi_{(u,v)}$ is transverse to conjugation. Let us consider two cocycles 
 \begin{align*}
     d_1(\gamma)=\frac{d}{du}\Bigg|_{u=0}\Psi_{(u,0)}(\gamma)\rhp(\gamma)^{-1}, \quad d_2(\gamma)=\frac{d}{dv}\Bigg|_{v=0}\Psi_{(0,v)}(\gamma)\rhp(\gamma)^{-1}, \qquad \forall\gamma\in \Gamma.
 \end{align*}
 They are determined by the images of the generators of $\Gamma$, so we compute the following derivatives:
 \begin{align*}
     &\frac{d}{du}\Bigg|_{u=0}\Psi_{(u,0)}(a)=0, \qquad \frac{d}{du}\Bigg|_{u=0}\Psi_{(u,0)}(a^2b)=
      \begin{bmatrix}
 0& 1/2 & 0&1/6  \\
 & 1 &  &  1/2 \\
 &  & -1 &0\\
 &  &  &0 
\end{bmatrix},\\
& \frac{d}{dv}\Bigg|_{v=0}\Psi_{(0,v)}(a)=0, \qquad \frac{d}{dv}\Bigg|_{v=0}\Psi_{(0,v)}(a^2b)=  \begin{bmatrix}
0 & 0 & 1/2& 0 \\
 &  & 1 &0  \\
 & 1&  &1/2 \\
 0&  &  & 0
\end{bmatrix}.
 \end{align*}
 By construction, $d_1,d_2$ are tangent to $\mathcal{S}_{\rho} :=\Psi_{(\R^2)}$ at $\rhp$, so it suffices to see that they span a 2-dimensional vector space of cocycles which are not inner. We will proceed by contradiction. First, assume $d_1$ is a coboundary. Then, there is an element $v\in\slr$ that fulfills
 \begin{equation}\label{conduno}
     v\in \ker(\Arh(a)-\id)=\left\{\begin{bmatrix}
 a& &  & b\\
& x & y &  \\
& -y &x  & \\
 c&  & & -(a+2x) 
\end{bmatrix}:a,b,c,x,y\in\R\right\}
 \end{equation}
 and
 \begin{equation}
     d_1(a^2b)= (\Arh(a^2b)-\id) v.\label{pt}
 \end{equation}
 Now, given $m\in\mathfrak{gl}_4\R$, let us denote by $(\cdot)_{2,3}$ the $2\times 2$ principal submatrix consisting of elements in rows and columns $2$ and $3$. The condition in Equation \eqref{conduno} implies that, for any such $m\in \mathfrak{gl}_4\R$, we have
 \begin{equation}\label{dostres}
     (vm)_{2,3}=(v)_{2,3}(m)_{2,3},\text{ and} \quad  (mv)_{2,3}=(m)_{2,3}(v)_{2,3}.
 \end{equation}
 Moreover, condition in (\ref{pt}) reads as
 \begin{equation}\label{conddos}
      \rhp(a^2b) v-v\rhp(a^2b)=\frac{d}{du}\Bigg|_{u=0}\Psi_{(u,0)}(a^2b).
 \end{equation}
 Further, Equation (\ref{dostres}) together with Equation (\ref{conddos}) implies
 \begin{equation}\label{auxiliar}
\left(\rhp(a^2b)\right)_{2,3}(v)_{2,3}-(v)_{2,3}\left(\rhp(a^2b)\right)_{2,3}=\begin{bmatrix}
  1 &   \\
 & -1
\end{bmatrix}\neq 0\in \mathrm{GL}_2\R.
 \end{equation}
However, $ \left(\rhp(a^2b)\right)_{2,3}=\id\in\mathrm{GL}_2\R$, so (\ref{auxiliar}) gives rise to a contradiction. A very similar argument carries out the proof of $d_2$ being not an inner cocycle, and in fact for any nontrivial linear combination $\lambda_1d_1+\lambda_2 d_2$ being not an inner cocycle. 
\par We also need to show that any projective holonomy close to the hyperbolic one lies on $\Ad(G)(\mathcal{S}_{\rhp})$, the set of conjugacy classes having a representative in the slice. To do this, it suffices to see that $\Ad:\Slr\times \mathcal{S}_{\rhp} \rightarrow \Rep(\Gamma,\Slr)$ is a submersion. Differentiation at $(\id,\rhp)$ yields
\begin{align*}
    \slr \oplus T_{\rhp}\mathcal{S}_{\rhp}&\rightarrow T_{\rhp}\Rep(\Gamma,\Slr)\\
    (v,z)&\rightarrow v-\Arh(v)+z,
\end{align*}
Hence, $d\Ad_{(\id,\rhp)}(v+z)=0$ implies $z=\Arh(v)-v\in B^1(\Gamma,\Arh)$. Definition \ref{slice} requires that the unique inner cocycle tangent to an slice is the trivial one, so
\begin{equation*}
\ker (d\Ad)_{(\id,\rhp)}\cong H^0(\Gamma,\Arh),
\end{equation*}
which is by equality (\ref{eq10coh}) in Proposition \ref{0coc} a 1-dimensional vector space. Thus,
\begin{equation*}
    \dim{\Im (d\Ad)_{(\id,\rhp)}}=\dim \left(\slr\oplus T_{\rhp}S\right)-\dim \ker (d\Ad)_{(\rhp,\id)}= 15+2-1=16,
\end{equation*}
which equals to $\dim T_{\rhp}\Rep(\Gamma,\Slr)=\dim \Rep_0(\Gamma,\Slr)$ by Lemma \ref{lema}, as desired.
\par Finally, since $\rho_{(u,v)}|_{\Gamma_0}$ is either diagonalizable or hyperbolic, the fact that any holonomy $\rho$ near $\rhp$ is conjugate to someone in $\mathcal{S}_{\rhp}$ implies that any such $\rho$ is either diagonalizable or hyperbolic.
\end{proof} 
\begin{rmk}\label{extra}
    The holonomies we refer to in Theorem \ref{main} will be shown to be actually holonomies of properly convex projective structures at the proof of Theorem \ref{thm}. This provides examples of realization of the claim that \textit{generalized cusps are all deformations of hyperbolic cusps} (see Ballas-Cooper-Leitner \cite{MR4125754}), since diagonal holonomies correspond to generalized cusps of $C_3$ type.
\end{rmk}
\section{Deformations of hyperbolic 3-orbifolds with turnover cusps}\label{3}
In the setting or Riemannian geometry, it is clear what is the meaning of a submanifold to be \textit{totally geodesic}. We generalize this concept following \cite{MR3780442} and \cite{MR1053346}.
\begin{de}[Totally geodesic]\label{tg}
    A boundary component of a projective n-orbifold is \textbf{totally geodesic} if its image by the developing map is contained in $\R\mathbb{P}^{n-1}$. 
\end{de}
Now, we want to characterize projective deformations of hyperbolic 3-orbifolds whose ends have turnover cross section. Any end of such an orbifold is a product $[0,\infty)\times \Orb^2$, where $\Orb^2$ is either a compact Euclidean turnover (so we get a finite volume end, or a \textit{cusp}) or a hyperbolic turnover (so the end is a \textit{funnel} and has infinite volume). We can alternatively include these orbifolds in the setting of finite volume by deformation retracting any funnel to a boundary component. Notice that we can do this because hyperbolic turnovers are always totally geodesic in a hyperbolic 3-orbifold. On the other hand, the list of compact Euclidean turnovers is short, and there are only the subsequent three: $S^2(2,3,6)$, $S^2(2,2,4)$, and $S^2(3,3,3)$. Let us see that the first two of them yield projectively rigid ends, so any flexible cusp is of the form $[0,\infty)\times S^2(3,3,3)$.
\begin{lmm}\label{ultimolem}
    The hyperbolic cusps $[0,\infty)\times S^2(2,3,6)$ and $[0,\infty)\times S^2(2,2,4)$ are projectively rigid.
\end{lmm}
\begin{proof}
    We shall prove that the first cohomology at the hyperbolic holonomy $\mathrm{h_{hyp}}$ is trivial, which by Weil's theorem (\cite{MR169956}) implies local rigidity.
    \par Let $\Orb^2$ be any of $S^2(2,2,4)$ or $S^2(2,3,6)$. The stabilizer at any branching point is given by a cyclic group of rotations acting on $C_0$ as we saw for $S^2(3,3,3)$ (see equations \eqref{g1} and \eqref{a}). The key point is that the action of such groups in $\slr$ stabilizes a 5-dimensional vector subspace regardless of the (finite) order of the group \textit{unless} it is $2$. Namely, if $\gamma\in\pi_1\Orb^2$ has order two, $\mathrm{h_{hyp}}(\gamma)=1\oplus-\id\oplus 1$ acts trivially on the $2\times 2$ principal submatrix consisting of elements in rows and columns 2 and 3 of any matrix in $\slr$, and then preserves a 7-dimensional subspace of $\slr$. Thus, for the twisted Euler characteristic (see Definition 3.1 in \cite{MR4495671}) we get
    \begin{equation*}
        \tilde{\chi}(\Orb^2,\Ad\mathrm{h_{hyp}})=\chi(\Orb^2\setminus\Sigma)\dim \Slr+\sum_{v\in \Sigma} \dim\left(H^0(\Stab(v),\slr)\right)=-15+2\cdot 5+7=2.
    \end{equation*}
    Therefore, 
    $$\dim H^0(\Orb^2,\Ad\mathrm{h_{hyp}})-\dim H^1(\Orb^2,\Ad\mathrm{h_{hyp}})+\dim H^2(\Orb^2,\Ad\mathrm{h_{hyp}})=\tilde{\chi}(\Orb^2,\Ad\mathrm{h_{hyp}})=2$$
    and, by Poincaré duality, we obtain
    $$\dim H^1(\Orb^2,\Ad\mathrm{h_{hyp}})=2\left(\dim H^0(\Orb^2,\Ad\mathrm{h_{hyp}})-1\right).$$
    Formal computations yield $\dim H^0(\Orb^2,\Ad\mathrm{h_{hyp}})=1$, so we get $\dim H^1(\Orb^2,\Ad\mathrm{h_{hyp}})=0$, as claimed.
\end{proof}
Our main result in this section reads as follows:
\begin{thm}
    Let $\Orb^3$ be a topologically finite hyperbolic 3-orbifold whose ends have turnover cross section. Then, any projective deformation of $\Orb^3$ is a properly convex projective 3-orbifold whose ends are either hyperbolic cusps or totally geodesic.\label{thm}
\end{thm}
\begin{proof}
    First, let us analyze the deformations of the funnels. Corollary 4.4 in Porti \cite{porti} tells us that any $\Slr$ holonomy $\rho$ of a hyperbolic turnover is conjugate to a holonomy of the form $\tilde{\rho}\oplus 1$, where $\tilde{\rho}$ is a $\mathrm{SL_3}\R$ holonomy representation. Thus, since hyperbolic turnovers are totally geodesic in the hyperbolic structure of $\Orb^3$, we deduce that they remain totally geodesic in the deformed projective orbifold (in the broader sense of Definition \ref{tg}). 
    \par Now, we focus on a flexible cusp $C$. We follow the notation of Section \ref{2}, and also of the proof of Theorem \ref{main}. The geometric interpretation of the exact sequence (\ref{exsq}) implies that the deformation of an (orbifold) cusp $C=[0,\infty)\times S^2(3,3,3)$ of fundamental group $\Gamma$ yields a deformation of its (manifold) cusp covering $\Tilde{C}=[0,\infty)\times T^2$ with fundamental group $\Gamma_0$. We also know that any nontrivial deformation $\delta_{(t,\theta)}\Tilde{C}$ is a projective generalized cusp with diagonalizable holonomy, by Theorem \ref{main}. This situation is studied in Section 5 of \cite{MR3780442}, where it is proved that any such $\delta_{(t,\theta)}\tilde{C}$ is totally geodesic. Indeed, the authors of \cite{MR3780442} show that there is a properly convex domain $\Omega_{(t,\theta)}$ invariant by the action of $$\Psi_{(t\cos{3\theta},t\sin{3\theta})}(\Gamma_0)$$ such that its boundary $\partial \Omega_{(t,\theta)}$ contains both a triangle
    $$T_{(t,\theta)}:=\triangle\left(p_1(t,\theta),p_2(t,\theta),p_3(t,\theta)\right)$$
    and a fixed point $p_{\infty}$ dual to the projective plane containing $T_{(t,\theta)}$. We shall see that $\delta_{(t,\theta)}C$ is itself a totally geodesic (orbifold) generalized cusp. To do it, note that $\Gamma=<\Gamma_0,a>$ and also that $\Psi_{(t\cos{3\theta},t\sin{3\theta})}(a)$ preserves the triangle $T_{(t,\theta)}$ and the point $p_{\infty}$ for any $t>0$. Thus, the quotient $$\mfaktor{<\Psi_{(t\cos{3\theta},t\sin{3\theta})}(a)>}{\delta_{(t,\theta)}\Tilde{C}}$$
    is a totally geodesic (orbifold) generalized cusp.
    \par Finally, we apply the holonomy principle in Theorem 0.2 of Cooper-Long-Tillmann \cite{MR3780436} (see also Theorem 2.3 in \cite{MR3780442}) to ensure that deformations of properly convex orbifolds whose ends are generalized cusps produce again properly convex orbifolds \textit{provided that the generalized cusps are deformed to generalized cusps}. To argue in such a way with the projective orbifold $\Orb^3_\Pp$ obtained by deforming $\Orb^3$, we obtain a finite volume properly convex orbifold $2\Orb^3_\Pp$ whose ends are generalized cusps by doubling along the hyperbolic turnovers, which are totally geodesic. Then, we conclude that the projective orbifold $\Orb^3_\Pp$ is itself properly convex. Thus, any projective deformation of $\Orb^3$ is a properly convex projective orbifold whose ends are hyperbolic cusps (if the sectional turnover is not $S^2(3,3,3)$, or if the deformation is trivial) or totally geodesic (the case above of nontrivial deformation of a cusp $C=[0,\infty)\times S^2(3,3,3)$, or any funnel of hyperbolic turnover cross section).
\end{proof}
\begin{rmk}
    Notice that there are no assumptions about the infinitesimal deformations of the hyperbolic 3-orbifold $\Orb^3$, in contrast to \cite{MR3780442}.
\end{rmk}
Now, we discuss a family of examples to illustrate Theorem \ref{thm}.
\begin{ex}
\label{ex:Opqr}
    Consider the family of hyperbolic 3-orbifolds $\{\Orb^3_{p,q,r}:p,q,r\geq 3\}$ in Figure~\ref{Fig:exs} with orbifold fundamental group $$\pi_1\left(\Orb^3_{p,q,r}\right):=<a,b,c|a^p=b^q=c^r=1,c=aba^{-1}b^{-1}>,$$
    which has two ends with orbifold fundamental groups
    $$<a,c|a^p=(a^{-1}c)^p=c^r=1>\cong \pi_1\left(S^2(p,p,r)\right)$$
    and
    $$ <b,c|b^q=(cb)^q=c^r=1>\cong \pi_1\left(S^2(q,q,r)\right),$$
    respectively.
    Therefore, both ends have turnover cross section, and we may apply Theorem \ref{thm}.
    \par For $p=q=r=3$, the projective structure moduli space was studied in Porti-Tillmann \cite{MR4264240}. This $\Orb_{(3,3,3)}^3$ is a finite volume hyperbolic 3-orbifold whose two cusps are $[0,\infty)\times S^2(3,3,3)$. Thus, Theorem \ref{thm} tells us that any projective deformation of the hyperbolic structure of $\Orb^3_{3,3,3}$ is a properly convex projective structure with hyperbolic cusps or totally geodesic ends.
    \par For $pqr>3^3$, the orbifold $\Orb^3_{(p,q,r)}$ remains hyperbolic but we do not get finite volume anymore. If $pr,qr>3$, then $\Orb^3_{(p,q,r)}$ has two turnover funnels, $[0,\infty)\times S^2(p,p,r)$ and $[0,\infty)\times S^2(q,q,r)$, and no cusps. Thus, after any deformation of the hyperbolic structure, $\Orb^3_{(p,q,r)}$ remains a properly convex projective orbifold with totally geodesic ends. If $pr=3^2$, then $\Orb^3_{(p,q,r)}$ is a hyperbolic 3-orbifold with one funnel $[0,\infty)\times S^2(q,q,r)$ and one cusp $[0,\infty)\times S^2(3,3,3)$. Therefore, the projectively deformed orbifold is properly convex with two ends, one of them a totally geodesic projective funnel and the other one a generalized cusp which is either hyperbolic or totally geodesic with diagonal holonomy.
\end{ex}
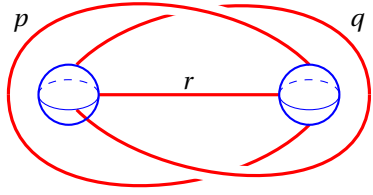
\begin{figure}[h]
	\begin{center}
		\begin{tikzpicture}[scale=.8]
			\begin{scope}[shift={(10,0)}]
				\draw [very thick, color=red] (-1.5,0)--(1.5,0);
				\draw [very thick, color=red] (-1.866, .5) to[out=45, in=90] (3,0);
				\draw [line width = 2.5mm, color=white](2, .5) to[out=135, in=90] (-3,0);
				\draw [very thick, color=red] (2, .5) to[out=135, in=90] (-3,0);
				\draw [very thick, color=red] (2, -.5) to[out=-135, in=-90] (-3,0);
				\draw [line width = 2.5mm, color=white] (-1.866, -.25) to[out=-45, in=-90] (3,0);
				\draw [very thick, color=red] (-1.866, -.25) to[out=-45, in=-90] (3,0);
				\draw (-2.8,1.2) node{$p$};
				\draw (2.8,1.2) node{$q$};
				\draw (0,0.2) node{$r$};
				\draw[thick, blue, opacity=.6] (-2,0) circle [radius=.5];
				\draw[blue, thin, opacity=.6] (-2.5,0) arc[ x radius=.5, y radius=.25, start angle=180, end angle=360];
				\draw[blue, thin, dashed, opacity=.6] (-2.5,0) arc[ x radius=.5, y radius=.25, start angle=180, end angle=0];
				\draw[blue, thin, opacity=.6] (1.5,0) arc[ x radius=.5, y radius=.25, start angle=180, end angle=360];
				\draw[blue, thin, dashed, opacity=.6] (1.5,0) arc[ x radius=.5, y radius=.25, start angle=180, end angle=0];
				\draw[thick, blue, opacity=.6] (2,0) circle [radius=.5];
			\end{scope}
		\end{tikzpicture}
	\end{center}
	\caption{The family of orbifolds $\Orb^3_{p,q,r}$, with $p,q,r\geq 3$, in Example~\ref{ex:Opqr}.
	}
	\label{Fig:exs}       
\end{figure}
\section{Global properties of the character variety $X(\Gamma,\Slr)$}\label{4}
Recall that $\Gamma:=\pi_1\left(S^2(3,3,3)\right)$. Throughout this section, we deal with the global problem of projective structures on the orbifold cusp $[0,\infty)\times S^2(3,3,3)$. There is a straightforward description due to Lawton (see \cite{MR2329569}) of the $\mathrm{SL}_3\C$ character variety of the free group with two generators. Then, our first goal is to find the relation between the $\mathrm{SL}_3\C$ and $\Slr$ character varieties of $\Gamma$.
\par First, we address the relation between the $\mathrm{SL}_n\R$ character variety and the points of the $\mathrm{SL}_n\C$ character variety having a real semisimple representation in their fiber.
\begin{lmm}\label{lmm1}
    Let $G$ be a finitely generated group, and $n\in\N$ an odd number. Then, any $\mathrm{SL_n\C}$ conjugacy class of semisimple representations having a real representative yields a unique $\mathrm{SL_n\R}$ conjugacy class. Thus,
    there is a continuous injection
    $$[i_*]:X(G,\mathrm{SL_n\R)}\to X(G,\mathrm{SL_n\C)}$$
    verifying $[i_*]\circ\pi_{\R}=\pi_{\C}\circ i_*$, where $i_*$ is given by postcomposition with $i:\mathrm{SL_n\R}\hookrightarrow\mathrm{SL_n\C}$ and $\pi_{\K}$ are the quotient projections from the $\mathrm{SL_n}\K$ representation variety to the $\mathrm{SL_n}\K$ character variety, for $\K\in\{\R,\C\}$.
\end{lmm}
\begin{proof}
Since $n$ is odd, then $$\mathrm{GL_n}\K\to\mathrm{SL_n}\K,M\to\frac{1}{\sqrt[n]{\det M}}M$$ is well-defined. Thus, the conjugacy action of $\mathrm{GL_n\R}$ and $\mathrm{SL_n}\R$ coincides. Hence, $$\Rep^{SS}(G,\mathrm{SL_n}\R)/\mathrm{SL_n}\R\cong \Rep^{SS}(G,\mathrm{SL_n}\R)/\mathrm{GL_n}\R.$$ 
Since $\Rep^{SS}(G,\mathrm{SL_n}\R)/\mathrm{GL_n}\R$ is a subset of $\Rep^{SS}(G,\mathrm{GL_n\R})/\mathrm{GL_n\R}$, it suffices (actually, it is more) to show the lemma for $$[i_*]:X(G,\mathrm{GL_n\R)}\to X(G,\mathrm{GL_n\C)}.$$
    \par Now, notice that if $\rho\in\Rep(G,\mathrm{GL_n\R)}$ is semisimple, then $i_*\rho\in \Rep(G,\mathrm{GL_n\C)}$ is semisimple (because any $\rho$-invariant subspace either is itself $i_*\rho$-invariant or splits into a direct sum of $i_*\rho$-invariant subspaces). Thus, there is a well-defined continuous injection $$i_*:\Rep^{SS}(G,\mathrm{GL_n\R)}\to \Rep^{SS}(G,\mathrm{GL_n\C)}.$$
    Further, both projections $\pi_{\R},\pi_{\C}$ are continuous and surjective. Hence, we just need to check injectivity, i.e. that two different $\mathrm{GL_n}\R$ classes of real representations yield two different $\mathrm{GL_n}\C$ classes.
    \par Let $\rho_1,\rho_2$ two real semisimple representations such that $i_*\rho_1,i_*\rho_2$ lie in the same $\mathrm{GL_n}\C$-class. Then, there is a complex matrix $M\in \mathrm{GL_n}\C$ such that 
    \begin{equation}\label{reps12}
       M\rho_1(g)=\rho_2(g) M, \quad\forall g\in G.
    \end{equation}
    Taking real and imaginary parts in Equation \eqref{reps12}, we get
    \begin{align*}
      & \Re(M)\rho_1(g)=\rho_2(g) \Re(M), \quad\forall g\in G,\\
       & \Im(M)\rho_1(g)=\rho_2(g) \Im(M), \quad\forall g\in G.
    \end{align*}
    Thus, for any real matrix $M'$ in the real vector space of $n\times n$ matrices generated by $<\Re(M),\Im(M)>$, the equality $M'\rho_1(g)=\rho_2(g) M'$ holds for any $g\in G$. Then, it suffices to show that $<\Re(M),\Im(M)>$ contains some automorphism $M'$, because then $M'\in\mathrm{GL_n\R}$ conjugates $\rho_1$ into $\rho_2$, so $\pi_\R(\rho_1)=\pi_\R(\rho_2)$. To do that, let us define the polynomial with real coefficients
    $$P(x):=\det\left(\Re(M)+x\Im(M) \right)\in\R[x].$$
    Since $P(i)=\det(M)\neq 0$, then $P\not\equiv 0$, so it must exist some $x_0\in\R$ (actually, any real number but finitely many fits) such that $P(x_0)\neq 0$. Therefore, $\Re(M)+x_0\Im(M)$ is a $\mathrm{GL_n}\R$ matrix conjugating $\rho_1$ to $\rho_2$, as desired.
    \end{proof}
    \begin{rmk}
    When $n$ is even, Lemma \ref{lmm1} fails. However, it still holds when $n$ is even if we work with $\mathrm{SL_n\R}^\pm$ and $\mathrm{SL_{n}\C}^\pm$ (instead of $\mathrm{SL_n\R}$ and $\mathrm{SL_{n}\C}$, respectively). 
    \end{rmk}
     \begin{rmk}
    Lemma \ref{lmm1} also applies to orbit spaces of representations, meaning that $i_*$ induces an injective function $$[i_*]_{\mathrm{orb}}:\Rep(G,\mathrm{SL_n\R})/\mathrm{SL_n}\R\to \Rep(G,\mathrm{SL_n\C)}/\mathrm{SL_n\C}.$$
    This is consequence of the fact that the proof does not make use of the semisimplicity of representations.
    \end{rmk}
    Once Lemma \ref{lmm1} is established, we need to find the relation between $X(\Gamma,\Slr)$ and $X(\Gamma,\mathrm{SL}_3\R)$. This relation is suggested by Theorem \ref{main}, which describes a slice around $\rhp\in\Rep(\Gamma,\Slr)$ containing only reducible representations.
    \begin{prp}
      \label{nonsimple}
        Any character $\chi\in X(\Gamma_,\mathrm{SL_4\R})$ is not simple. Moreover,
        \begin{equation}\label{s}
            X(\Gamma_,\mathrm{SL_4\R})=i_*\left(\Rep^{SS}(\Gamma_,\mathrm{SL_3\R})\right)/\Slr\cup j_*\left(\Rep^{SS}(\Gamma_,\mathrm{SL_2\R}\times \mathrm{SL_2\R})\right)/\Slr,
        \end{equation}
        where $i:\mathrm{SL_3\R}\hookrightarrow\Slr$ and $j:\mathrm{SL_2\R}\times \mathrm{SL_2\R}\hookrightarrow\Slr$ are the natural inclusions, and the star subindex means postcomposition. 
    \end{prp}
    \begin{proof}
    Notice that the relations $a^3=b^3=(ab)^3=\id$ force any semisimple $\rho:\Gamma\to \Slr$ fixing a subspace $V<\R^4$ to act on it by $\mathrm{SL(V)}$ automorphisms, i.e. having determinant equal to one. Notice that $\mathrm{SL_1(\R)}$ is just the identity. Then, it suffices to prove that any semisimple representation $\Gamma\to \Slr$ fixes a line or a plane to get Equation \eqref{s}.
    \par The key fact we shall use is that the \textit{twisted Euler characteristic} of an orbifold (Definition 3.1 in \cite{MR4495671}) can be computed taking into account no more than the topology of the space of regular points and the twisted zero cohomology at the stabilizer of each branching point. Even more, since $S^2(3,3,3)$ has the same stabilizer at its three branching points, $<a>\cong \Z_3$, we just need to analyze the image $\rho(a)$ to compute the twisted Euler characteristic. Since $a$ has order three, then $\rho(a)$ must be conjugate to one of
     $$\id,P_3\oplus 1,R\oplus R,$$
     where $P_3$ is a $3\times 3$ permutation matrix of order three and $R$ is a $2\times 2$ rotation matrix of angle $\frac{2\pi}{3}$. We shall study these three cases separately.
     \par If $\rho(a)$ is conjugate to $\id$ (so trivially $\rho(a)=\id)$, then $\rho(a^2b)=\rho(ba^2)$ and $\rho$ fixes any eigenspace of $\rho(a^2b)$. Thus, $\rho$ fixes a line or a plane, and the proposition is proved for this case.
     \par Now, assume that $\rho(a)$ is conjugate to $P_3\oplus 1$. The symmetries of this matrix suggest that $\rho$ fixes a line, so we calculate $H^0\left(\Gamma,\R^4\right)$. To do so, we compute the twisted Euler characteristic:
     $$\tilde{\chi}\left((S^2(3,3,3),\R^4\right)=\chi\left(S^2\setminus\{*\}_{i=1}^3\right)\dim\R^4+3\dim H^0(<a>,\R^4).$$
     Since the trice punctured sphere $S^2\setminus\{*\}_{i=1}^3$ has Euler characteristic $-1$ and $H^0(<a>,\R^4)\cong\R^2$, we obtain:
      $$\tilde{\chi}\left((S^2(3,3,3),\R^4\right)=-4+3\cdot 2=2.$$
    Hence,
    \begin{equation*}
        \dim H^0(\Gamma,\R^4)-\dim H^1(\Gamma,\R^4)+\dim H^2(\Gamma,\R^4)=\tilde{\chi}\left((S^2(3,3,3),\R^4\right)=2,
    \end{equation*}
    and by Poincaré duality we get
     \begin{equation}  
    \dim H^0(\Gamma,\R^4)\geq 1.
    \end{equation}
    Thus, $\rho$ fixes a line, as desired. Therefore, any character represented by $\rho$ is in $i_*\left(\Rep^{SS}(\Gamma_,\mathrm{SL_3\R})\right)/\Slr$.
    \par Now, assume that $\rho(a)$ is conjugate to $R\oplus R$. The symmetries of this matrix suggest that $\rho$ fixes a plane. Indeed, we shall show that $\rho$ fixes a point of the Grassmannian of planes, $\mathrm{Gr_2(\R^4)}$. To do so, consider $\mathrm{Gr_2(\R^4)}$ as a projective hypersurface of $\Pp(\Lambda^2\R^4)$ via Plücker embedding, where $\Lambda^2\R^4$ denotes the vector space of 2-forms in $\R^4$. We consider the natural action $\Gamma\acts \Lambda^2\R^4$ via $(\rho\otimes\rho)|_{\Lambda^2\R^4}$, and study the dimension of $H^0(\Gamma,\Lambda^2\R^4)$. 
    \par First, we compute $H^0(<a>,\Lambda^2\R^4)$. Let us denote by $\omega\neq 1$ a complex cubic root of unity and by $\{e_i\}_{i=1}^4$ a basis of $\C^4$ such that $\rho(a)=\mathrm{diag}(\omega,\omega,\omega^2,\omega^2)$ (we abuse the notation and use $\rho$ instead of $z_*\rho$, where $z:\Slr\hookrightarrow\mathrm{SL_4\C}$ is the natural inclusion). Then,
    $$H^0(<a>,\Lambda^2\C^4)\cong<e_1\wedge e_3,e_1\wedge e_4,e_2\wedge e_3,e_2\wedge e_4>\cong \C^4.$$
    Thus, $H^0(<a>,\Lambda^2\R^4)\cong\R^4$ and we obtain
    $$\tilde{\chi}\left((S^2(3,3,3),\Lambda^2\R^4\right)=-6+3\cdot 4=6,$$
    where we have used that $\dim \Lambda^2\R^4=\frac{4(4-1)}{2}=6$. Hence,
     \begin{equation*}
        \dim H^0(\Gamma,\Lambda^2\R^4)-\dim H^1(\Gamma,\Lambda^2\R^4)+\dim H^2(\Gamma,\Lambda^2\R^4)=\tilde{\chi}\left((S^2(3,3,3),\Lambda^2\R^4\right)=6,
    \end{equation*}
    and by Poncaré duality,
    \begin{equation}  
    \dim H^0(\Gamma,\Lambda^2\R^4)\geq 3.
    \end{equation}
    Thus, the projectivization $\Pp\left(H^0(\Gamma,\Lambda^2\R^4)\right)$ is a projective space of dimension at least $2$. It follows from Bézout theorem that any projective plane meets any projective hypersurface, so the intersection $\Pp\left(H^0(\Gamma,\Lambda^2\R^4)\right)\cap \mathrm{Gr}_2(\R^2)$ cannot be empty, meaning that $\rho$ preserves at least one plane. Therefore, a character represented by $\rho$ lies on $j_*\left(\Rep^{SS}(\Gamma_,\mathrm{SL_2\R}\times \mathrm{SL_2\R})\right)/\Slr$. 
    \end{proof}
    \begin{lmm}\label{lemasln}
 Let $G$ be a finitely generated group, and let $n\in\N$ an odd number. Then, any $\mathrm{SL_{n+1}\R}$ conjugacy class of semisimple representations having a line of fixed points yields a unique $\mathrm{SL_n\R}$ conjugacy class. Thus,
    there is a continuous injection
    $$[i_*]:X(G,\mathrm{SL_n\R})\to X(G,\mathrm{SL_{n+1}\R})$$
    verifying $[i_*]\circ\pi_{n}=\pi_{n+1}\circ i_*$, where $i_*$ is given by postcomposition with $i:\mathrm{SL_{n}}\R\hookrightarrow\mathrm{SL_{n+1}\R}$, and $\pi_{k}$ is the quotient projection from the $\mathrm{SL_k}\R$ representation variety to the $\mathrm{SL_k}\R$ character variety.
    \end{lmm}
    \begin{proof}
    Continuity is clear, so we only need to check that any two representations $$i_*\rho_1,i_*\rho_2\in i_*\left(\Rep^{SS}(G,\mathrm{SL_n\R})\right)$$ which are $\mathrm{SL_{n+1}\R}$ conjugate come from $\mathrm{SL_n\R}$ conjugate representations $$\rho_1,\rho_2\in \Rep^{SS}(G,\mathrm{SL_n\R}).$$ 
    \par By considering $i_*\rho_1,i_*\rho_2$ as complex representations, we can assume that, if they are conjugate (by an $(n+1)\times(n+1)$ complex matrix), then their characters coincide. It is equivalent to say that the characters of $\rho_1$ and $\rho_2$ coincide, since $\chi(i_*\rho)=\chi(\rho)+1$ for any $\rho\in\Rep(\Gamma,\mathrm{SL_n\C})$. Thus, there is a complex $n\times n$ matrix that conjugates $\rho_1$ and $\rho_2$ (see Theorem 1.28 in Lubotzky-Magid \cite{MR818915}). Since $n$ is odd, we may apply Lemma \ref{lmm1} to conclude that there is a matrix in $\mathrm{SL_n\R}$ conjugating $\rho_1$ to $\rho_2$.
        \end{proof}
         \begin{rmk}
        Lemma \ref{lemasln} generalizes to even $n$ if we consider $\mathrm{SL}^\pm$ groups instead of $\mathrm{SL}$ groups.
    \end{rmk}
        We may improve the expression given in Proposition \ref{nonsimple} due to Lemma \ref{lemasln} as follows:
        \begin{cor}
        With the same notation as in Proposition \ref{nonsimple},
            \begin{equation}\label{eqcharavr}
            X(\Gamma_,\mathrm{SL_4\R})\cong X(\Gamma_,\mathrm{SL_3\R})\cup j_*\left(\Rep^{SS}(\Gamma_,\mathrm{SL_2\R}\times \mathrm{SL_2\R})\right)/\Slr.       \end{equation}\label{corolario}
        \end{cor}
    Now, we are able to state and show an algebraic description of $X(\Gamma,\Slr)$.
    \begin{thm}
        The character variety $X(\Gamma,\Slr)$ is a real algebraic set. It has a unique irreducible component $X_0\subset X(\Gamma,\Slr)$ of positive dimension, which is a real algebraic surface given implicitly in terms of traces of matrices by
        \begin{equation*}
              \tau^2-(rs-3)\tau+r^3+s^3-6rs+9=0.
        \end{equation*}
       $X_0$ contains exactly all the nontrivial characters that are the inclusion of simple $\mathrm{SL_3\R}$ characters plus (the inclusion of) one reducible $\mathrm{SL_3\R}$ character $\chp$. This $\chp$ is the unique singular point of $X_0$ and is the class of $\rhp$ when taking the real GIT quotient.
       Moreover, the complement of $X_0$ consists of $19$ isolated points having the following representatives:
      \begin{enumerate}
      \item $\rho(a^2b)=R\oplus R$, $\rho(a)\in\{\id\oplus\id,\id\oplus R^{\pm 1},R^{\pm 1}\oplus R^{\pm 1},R\oplus R^{-1}\}$,
            \item $\rho(a^2b)=R\oplus R^{-1}$, $\rho(a)\in\left\{\id\oplus\id,\id\oplus R^{\pm 1},R^{\pm 1}\oplus\id,\left(R\oplus R\right)^\pm,\left(R\oplus R^{-1}\right)^\pm\right\}$,
            \item $\rho(a^2b)=R\oplus \id$, $\rho(a)\in\{R^{\pm1}\oplus\id,\id\oplus\id\}$,
            \item $\rho(a^2b)=\id$, $\rho(a)=\id$,
        \end{enumerate}
        where $R$ is the rotation of angle $\frac{2\pi}{3}$.\label{charvar}
    \end{thm}
    \begin{proof}
    \par The fact that the character variety $X(\Gamma_,\mathrm{SL_4\R})$ is real algebraic follows from the description of it given in the statement. Thus, we shall prove that this description holds.
    \par Recall the description given by equation (\ref{eqcharavr}) in Corollary \ref{corolario}. We shall deal with both subsets $X(\Gamma_,\mathrm{SL_3\R})$ and $j_*\left(\Rep^{SS}(\Gamma_,\mathrm{SL_2\R}\times \mathrm{SL_2\R})\right)/\Slr$  separately.
    \par By Lemma \ref{lmm1}, we may analyze $X(\Gamma_,\mathrm{SL_3\R})$ by means of $\Re\left(X(\Gamma_,\mathrm{SL_3\C})\right)$. We know from \cite{MR2329569} that the $\mathrm{SL_3\C}$ character variety of the free group $F_2$ in two generators is a complex hypersurface of $\C^9$ given in terms of the traces of the products of the generators and its inverses, and their commutator. Regarding $F_2$ as an extension of $\Gamma$, we get nine complex coordinates for the ambient space of the $\mathrm{SL_3\C}$ character variety of $\Gamma$, namely:
    \begin{align*}
     x&:= \chi(a), & y&:=\chi(b), & z&:= \chi(ab),  & r&:= \chi(ab^{-1}), \qquad \tau:=\chi(aba^{-1}b^{-1}),\\
         u&:= \chi(a^{-1}), & v&:= \chi(b^{-1}), & w&:= \chi((ab)^{-1}), & s&:= \chi(a^{-1}b).
    \end{align*}
    A direct analysis yields four isolated points, corresponding to equivalence classes satisfying one of the following:
    \begin{itemize}
        \item $\rho(a)=\id,\rho(b)\neq \id$;
        \item $\rho(a)\neq \id,\rho(b)= \id$;
        \item $\rho(a)\neq \id,\rho(b)\neq\id,\rho(ab)=\id$;
        \item $\rho(a)=\id,\rho(b)=\id$.
    \end{itemize}      
    These four characters correspond to Case 3. and Case 4. in the statement of the theorem. 
    \par Now, let us analyze semisimple representations $\rho$ (with character $\chi$) such that $\rho(a),\rho(b),\rho(ab)$ are all nontrivial. Then, we get
    \begin{align*}
       & x=0,\quad y=0, \quad z=0, \\
        & u=0,\quad v=0, \quad w=0.
    \end{align*}
    Consequently, the corresponding characters lie on the algebraic subset $X_{0,\C}\subset \C^3$ given implicitly by 
    \begin{equation}\label{eqcompleja}
        \tau^2-(rs-3)\tau+r^3+s^3-6rs+9=0.
    \end{equation}
     Now, we need to distinguish the classes of real representations in $X_{0,\C}$. A necessary condition for $\rho$ to be real (maybe after a conjugation) is $\chi$ being real (note that any image of $\chi$ is a polynomial on the coordinates $r,s,\tau$, so $\chi\in\Re(X_{0,\C})$ reduces to $r,s,\tau\in\R$). In fact, Acosta proves in \cite{MR4027594} that $\chi$ being real is also sufficient to $\rho$ being real (up to conjugacy) if $\rho$ is simple. We claim that having real character is actually sufficient for any semisimple representation in our case, i.e. $X_0=\Re(X_{0,\C})$. 
    \par Let $\rho$ be semisimple and reducible, so it fixes a line. Thus, $\rho(a)$ cannot act as a permutation on the three invariant lines of $\rho(a^2b)$.
    Therefore, $\rho(a)$ acts trivially on $\mathrm{Spec}\left(\rho(a^2b)\right)$. Denote by $m\in\{1,2,3\}$ the number of different eigenvalues of $\rho(a^2b)$. Since $\rho(a)\neq \id$, then $m<3$. If $m=2$, semisimple reducible representations are given by
    $$\rho(a^2b)=\lambda\id\oplus \lambda^{-2},\quad \rho(a)=\mu R^{\pm1}\oplus \mu^{-2}$$
    where $\lambda,\mu$ are (complex) cubic roots of $1$, and $R$ is the rotation of angle $\frac{2\pi}{3}$. Hence, $\lambda^{-2}=\lambda$ and $\mu^{-2}=\mu$. However, the fact that $\chi(a^2b)=s$ is real implies that $\lambda=1$, which is absurd (it implies $m=1$). 
    Thus, $m=1$, so $\rho(a^2b)=\id$. We can conjugate $\rho(a)$ to be the permutation matrix of order three. Hence, the unique reducible character is represented by 
    $$\rho_{red}(a^2b)=\id,\rho_{red}(a)=1\oplus R,$$
    where $R$ is the rotation matrix of angle $\frac{2\pi}{3}$. This representation is real, so the claim is proved.
    \par We get also from the previous paragraph that there is just one reducible character $\chi_{red}\in X_0$, which has coordinates $(3,3,3)$. Direct calculations show the unique singular point of the real polynomial $$p(r,s,\tau)=\tau^2-(rs-3)\tau+r^3+s^3-6rs+9$$
    lying on $\Re(X_{0,\C})$ is precisely $(3,3,3)=\chi_{red}$.
    It remains to show that $\chi_{red}$ corresponds to $\rhp$, i.e., $i(\rho_{red})$ lies on the closure of the conjugacy orbit of $\rhp$, where $i:\mathrm{SL_3\R}\to\Slr$ is the natural inclusion. Let us define the family
     \begin{equation*}
         \left\{C_x:=x\oplus\id\oplus \frac{1}{x}\in \Slr: x>0\right\}.
     \end{equation*}
     Conjugating $\rhp$ by $C_x$ and then taking the limit as $x\to 0^+$, we obtain
     \begin{equation*}
         i\left(\rho_{red}(a^2b)\right)=\id=\lim_{x\to 0^+} C_x\rhp(a^2b) C_x^{-1}.
     \end{equation*}
     Further, $C_x$ leaves invariant $\rhp(a)=1\oplus R\oplus 1=i(\rho_{red})$. Therefore,
     \begin{equation*}
         \rho_{red}=\lim_{x\to 0^+} C_x\rhp C_x^{-1}.
     \end{equation*}
    \par Finally, we face to all the points of $X(\Gamma,\Slr)$ not coming from a $\mathrm{SL_3\R}$ character (so lying on $\Rep^{SS}(\Gamma_,\mathrm{SL_2\R}\times \mathrm{SL_2\R})/\Slr$ as said in Equation \ref{eqcharavr}). From the relations in $\Gamma$ we deduce that any representative $\rho$ of such characters must satisfy $\left(\rho(a^2b)\right)^3=\id$. Thus, $\rho$ splits in $\rho_1\oplus\rho_2$ where $\rho_1(\gamma),\rho_2(\gamma)$ are $\pm\frac{2\pi}{3}$ rotation matrices or $\id$, for any $\gamma\in\Gamma$. By a case-by-case analysis we obtain the $15$ isolated points written in 1. and 2. at the statement.
 \end{proof}
 \begin{prp}\label{4.9}
     No representation whose character lies in $X(\Gamma,\Slr)\setminus X_0$ is faithful, so the geometrically relevant component of the character variety is $X_0$.
 \end{prp}
    \begin{proof}
    Let us recall very briefly what we need from Section 2 of Böhm-Lafuente \cite{MR4420784}. Let $M\in\Slr$, and let $\Slr\cdot M$ the conjugation orbit of $M$. Then, the orbit $\Slr\cdot M$ is closed if and only if $M$ is semisimple, and in any case its closure $\overline{\Slr\cdot M}$ contains exactly a closed orbit. Moreover, if $M$ is semisimple, then any $M'\in\overline{\Slr\cdot M}$ has Jordan form $J'=M+N$, where $N$ is a nilpotent element of $\left(\slr\right)^{\Ad(M)}\cong H^0(<M>,\Ad)$. Thus, $M'$ is semisimple if and only if $N=0$, and the spectra of $M$ and $M'$ coincide. Now, the following lemma is clear:
    \begin{lmm}  \label{lm4.9} Let $\rho:\Gamma\to \Slr$ be a representation, and let $\gamma\in \Gamma$ be an arbitrary element of order three. Then:
    \begin{enumerate}
        \item If $\rho(\gamma)=\id$, then $\overline{\Slr\cdot\rho}$ does not contain any faithful representation.
        \item   If $\rho(\gamma)$ has order three, then any $\rho'\in \overline{\Slr\cdot\rho}$ satisfies $\rho'(\gamma)\in \Slr \cdot \rho(\gamma)$. 
    \end{enumerate}
    \end{lmm}
    Now, to be more readable, we repeat the list given in Theorem \ref{charvar} of $19$ semisimple representations corresponding to the $19$ isolated points that $X(\Gamma,\Slr)\setminus X_0$ consists of:
     \begin{enumerate}
      \item $\rho(a^2b)=R\oplus R$, $\rho(a)\in\{\id\oplus\id,\id\oplus R^{\pm 1},R^{\pm 1}\oplus R^{\pm 1},R\oplus R^{-1}\}$,
            \item $\rho(a^2b)=R\oplus R^{-1}$, $\rho(a)\in\{\id\oplus\id,\id\oplus R^{\pm 1},R^{\pm 1}\oplus\id,R^{\pm 1}\oplus R^{\pm 1},R^{\pm 1}\oplus R^{\mp 1}\}$,
            \item $\rho(a^2b)=R\oplus \id$, $\rho(a)\in\{R^{\pm1}\oplus\id,\id\oplus\id\}$,
            \item $\rho(a^2b)=\id$, $\rho(a)=\id$,
        \end{enumerate}
        where $R$ is the rotation of angle $\frac{2\pi}{3}$. It is clear that none of these representations is faithful. However, we also need to show that a faithful representation $\rho'$ cannot contain any $\rho$ in the list as a point of its orbit's closure. We shall proceed by elimination, that is, we take a faithful $\rho'$ and prove that no $\rho$ in the list fits.
    \par Using Lemma \ref{lm4.9}(1), we may discard any $\rho$ that lies in Cases 3. and 4. Thus, we assume that $\rho$ lies in Cases 1. or 2. Then, $\rho(a^2b)$ is a direct sum of rotations of order three. Since $a^2b$ has infinite and $a,b$ have order three, we may conclude that:
    \begin{itemize}
        \item  $\rho'(a^2b)$ does not respect the decomposition in two transverse planes $\pi_1\oplus\pi_2$ that $\rho$ respects;
        \item By Lemma \ref{lm4.9}(2), $\rho'(a)=\rho(a)$ and $\rho'(b)=\rho(b)$.
    \end{itemize}
    Since $\rho'(b)=\rho'(a)\rho'(a^2b)$, the two conditions above cannot be satisfied simultaneously, that is, we have eliminated all the candidates $\rho$ in the list. Therefore, the fibers of the real GIT quotient projection of the elements in the list contain only unfaithful representations.
    \end{proof}
    Now, we want to describe the topology of the irreducible component $X_0$. To do this, let us define $\tilde{X}_0<\Rep(\Gamma,\Slr)$ to be the smooth algebraic surface isomorphic to $\{x_1x_2x_3=1\}< \R^3$ whose points $\rho_{(x_1,x_2,x_3)}$ are given by 
\begin{equation*}
    \rho_{(x_1,x_2,x_3)}(a^2b)=i(x_1\oplus x_2\oplus x_3),\quad \rho_{(x_1,x_2,x_3)}(a)=i(P_3),
\end{equation*}
where $i:\mathrm{SL_3\R}\hookrightarrow\Slr$ is the natural inclusion and $P_3$ is the permutation matrix of order three. Notice that $\tilde{X}_0<\Rep^{SS}(\Gamma,\Slr)$. Further, there is a natural action of the symmetric group $\Sigma_3$ on $\tilde{X}_0$ induced by the permutation of the set $\{x_1,x_2,x_3\}$, i.e. the action identifies
$$\rho_{(x_1,x_2,x_3)}\equiv \rho_{(x_3,x_1,x_2)}\equiv \rho_{(x_2,x_3,x_1)}.$$
Moreover, this action preserves a unique connected component $\tilde{X}_0^+$ of $\tilde{X}_0$ given by  $$\tilde{X}_0^+:=\{\rho_{(x_1,x_2,x_3)}:x_1,x_2,x_3>0\}.$$
Using these ingredients, we get the following corollary: 
 \begin{cor}\label{generalization}
  There is an algebraic isomorphism $X_0\cong \tilde{X}_0/\Sigma_3$. Therefore, $X_0$ topologically splits into two connected components $S_1\sqcup S_2$ defined by
\begin{align*}
&S_1:=\tilde{X}_0^+/\Sigma_3, \\      &S_2:=\left(\tilde{X}_0\setminus\tilde{X}_0^+\right)/\Sigma_3.
\end{align*}
Both $S_1,S_2$ are homeomorphic to $\R^2$. Moreover, $S_1$ has an orbifold structure with a sole branching point $\chp$ of order three, while $S_2$ is smooth.
\end{cor}
\begin{proof}
    First, notice that $\tilde{X}_0\setminus\{\rho_{(1,1,1)}\}$ is the set of all semisimple representations which are the inclusion of nontrivial simple $\mathrm{SL_3\R}$ representations of $\Gamma$. Further, $\rho_{(1,1,1)}$ is a semisimple representative of $\chp$. Thus, we identify $X_0=\tilde{X}_0/\Slr$. 
    \par Secondly, note that $\rho_{(x_1,x_2,x_3)}$ is conjugate to $\rho_{(x_1',x_2',x_3')}$ if and only if 
    $$\{x_1,x_2,x_3\}=\{x_1',x_2',x_3'\}$$
    as sets. Thus, $\tilde{X}_0/\Sigma_3=\tilde{X}_0/\Slr=X_0$.
    \par Finally, we see that $\tilde{X_0}$ has four diffeomorphic (and diffeomorphic to $\R^2$) connected components:
    \begin{align*}
        & \tilde{X}_0^+=\left\{\rho_{(x_1,x_2,x_3)}\in \tilde{X}_0: x_1,x_2,x_3>0\right\},\\
        & \tilde{X}_0^1:= \left\{\rho_{(x_1,x_2,x_3)}\in \tilde{X}_0: x_1>0;x_2,x_3<0\right\},\\
        & \tilde{X}_0^2:=\left\{\rho_{(x_1,x_2,x_3)}\in \tilde{X}_0: x_2>0;x_1,x_3<0\right\}, \\
        & \tilde{X}_0^3:=\left\{\rho_{(x_1,x_2,x_3)}\in \tilde{X}_0: x_3>0;x_1,x_2<0\right\}.\\
    \end{align*}
    The symmetric group acts properly on $\tilde{X}_0^+$ and $\tilde{X}_0^1\sqcup\tilde{X}_0^2\sqcup\tilde{X}_0^3$ to get $S_1$ and $S_2$ (respectively). Thus, $S_1,S_2$ are both homeomorphic to $\R^2$. In the case of $\tilde{X}_0^+/\Sigma_3$, the action is not free at $\rho_{(1,1,1)}$, which is a branching point of order three, but it is free outside that point. This endows $S_1$ with a natural orbifold structure as in the statement. In the case of $\left(\tilde{X}_0^1\sqcup\tilde{X}_0^2\sqcup\tilde{X}_0^3\right)/\Sigma_3$, the action is free and identifies the three connected components, so $S_2$ is diffeomorphic to any $\tilde{X}_0^i$ (for $i=1,2,3$). Hence, $S_2$ is smooth. 
\end{proof}
Now, we describe the role of $\mathcal{S}_{\rhp}$, the slice around $\rhp$ given in Theorem \ref{main}, in the global structure of $\Rep(\Gamma,\Slr)$:
\begin{rmk}
    Let $\pi:\Rep(\Gamma,\Slr)\to X(\Gamma,\Slr) $ the (real GIT) quotient projection. Then, $$\pi(\mathcal{S}_{\rhp})=\pi(\tilde{X}_0^+).$$ Thus, $\pi|_{\mathcal{S}_{\rhp}}$ is a surjection onto $S_1$ which is 3-to-1 generically and 1-to-1 at the hyperbolic holonomy $\rhp$.
\end{rmk}
\begin{rmk}
    The topological description of $X_0$ given in Corollary \ref{generalization} can also be obtained in terms of the coordinates $r,s,\tau$ (see Theorem \ref{charvar}). Direct calculations yield
     \begin{align*}
         &S_1=X_0\cap \{r\geq 3,s\geq 3\},\\
         &S_2=X_0\setminus S_1.
     \end{align*}
\end{rmk}
\begin{rmk}
    As we said in Remark \ref{extra}, Ballas, Cooper and Leitner affirm in \cite{MR4125754} that generalized cusps are all deformations of hyperbolic cusps. Theorem \ref{thm} provides a converse for that statement in our case. Now, the disconnectedness $S_1\sqcup S_2$ of the algebraic component $X_0$ of all possibly geometric representations (see Proposition \ref{4.9}) obtained in Corollary \ref{generalization} can be interpreted as a global converse of Ballas-Cooper-Leitner's statement, since no representation in $S_2$ can be the holonomy of a generalized cusp of any type.
\end{rmk}
Finally, Corollary \ref{generalization} allows us to upgrade Theorem \ref{thm} to a global statement, not only concerning deformations but \textit{paths} of deformations:
\begin{cor}
    Let $\Orb^3$ be a topologically finite hyperbolic 3-orbifold whose ends have turnover cross section. Then, any path of projective deformations of $\Orb^3$ produces a properly convex projective 3-orbifold whose ends are either hyperbolic cusps or totally geodesic.
\end{cor}
\begin{proof}
    This corollary needs a more general holonomy principle that the one we invoked while proving Theorem \ref{thm}, which stated that, in the setting of finite volume convex projective orbifolds whose ends are generalized cusps, the set of properly convex projective structures is open. We refer to Theorem 0.2 in Cooper-Tillmann \cite{cooper2020spaceproperlyconvexstructures} to see that this set is also closed, so a connected component of the character variety. Notice that the reason why the theorem of Cooper-Tillmann applies is that we already known that all the representations in $\Rep(\Gamma,\Slr)$ are geometric (i.e., holonomies of convex projective structures).
\end{proof}

\begin{thebibliography}{10}
		
		\bibitem{MR4027594}
		Miguel Acosta.
		\newblock Character varieties for real forms.
		\newblock {\em Geom. Dedicata}, 203:257--277, 2019.
		
		\bibitem{MR4125754}
		Samuel~A. Ballas, Daryl Cooper, and Arielle Leitner.
		\newblock Generalized cusps in real projective manifolds: classification.
		\newblock {\em J. Topol.}, 13(4):1455--1496, 2020.
		
		\bibitem{MR3780442}
		Samuel~A. Ballas, Jeffrey Danciger, and Gye-Seon Lee.
		\newblock Convex projective structures on nonhyperbolic three-manifolds.
		\newblock {\em Geom. Topol.}, 22(3):1593--1646, 2018.
		
		\bibitem{MR4420784}
		Christoph B\"ohm and Ramiro~A. Lafuente.
		\newblock Real geometric invariant theory.
		\newblock In {\em Differential geometry in the large}, volume 463 of {\em
			London Math. Soc. Lecture Note Ser.}, pages 11--49. Cambridge Univ. Press,
		Cambridge, 2021.
		
		\bibitem{MR2043960}
		Suhyoung Choi.
		\newblock Geometric structures on orbifolds and holonomy representations.
		\newblock {\em Geom. Dedicata}, 104:161--199, 2004.
		
		\bibitem{MR3336086}
		D.~Cooper, D.~D. Long, and S.~Tillmann.
		\newblock On convex projective manifolds and cusps.
		\newblock {\em Adv. Math.}, 277:181--251, 2015.
		
		\bibitem{MR2264468}
		Daryl Cooper, Darren Long, and Morwen Thistlethwaite.
		\newblock Computing varieties of representations of hyperbolic 3-manifolds into
		{${\rm SL}(4,\Bbb R)$}.
		\newblock {\em Experiment. Math.}, 15(3):291--305, 2006.
		
		\bibitem{MR3780436}
		Daryl Cooper, Darren Long, and Stephan Tillmann.
		\newblock Deforming convex projective manifolds.
		\newblock {\em Geom. Topol.}, 22(3):1349--1404, 2018.
		
		\bibitem{cooper2020spaceproperlyconvexstructures}
		Daryl Cooper and Stephan Tillmann.
		\newblock The space of properly-convex structures.
		\newblock arXiv:2009.06568, 2020.
		
		\bibitem{MR1053346}
		William~M. Goldman.
		\newblock Convex real projective structures on compact surfaces.
		\newblock {\em J. Differential Geom.}, 31(3):791--845, 1990.
		
		\bibitem{MR4500072}
		William~M. Goldman.
		\newblock {\em Geometric structures on manifolds}, volume 227 of {\em Graduate
			Studies in Mathematics}.
		\newblock American Mathematical Society, Providence, RI, [2022] \copyright
		2022.
		
		\bibitem{MR2329569}
		Sean Lawton.
		\newblock Generators, relations and symmetries in pairs of {$3\times 3$}
		unimodular matrices.
		\newblock {\em J. Algebra}, 313(2):782--801, 2007.
		
		\bibitem{MR818915}
		Alexander Lubotzky and Andy~R. Magid.
		\newblock Varieties of representations of finitely generated groups.
		\newblock {\em Mem. Amer. Math. Soc.}, 58(336):xi+117, 1985.
		
		\bibitem{loh}
		Clara Löh.
		\newblock Group cohomology.
		\newblock
		\url{https://loeh.app.uni-regensburg.de/teaching/grouphom_ss19/lecture_notes.pdf},
		2019.
		
		\bibitem{MR4495671}
		Joan Porti.
		\newblock Dimension of representation and character varieties for two- and
		three-orbifolds.
		\newblock {\em Algebr. Geom. Topol.}, 22(4):1905--1967, 2022.
		
		\bibitem{porti}
		Joan Porti.
		\newblock Deforming reducible representations of surface and 2-orbifold groups.
		\newblock arXiv:2309.00282, 2023.
		
		\bibitem{MR4264240}
		Joan Porti and Stephan Tillmann.
		\newblock Projective structures on a hyperbolic 3-orbifold.
		\newblock {\em Acta Math. Vietnam.}, 46(2):347--355, 2021.
		
		\bibitem{MR1087217}
		R.~W. Richardson and P.~J. Slodowy.
		\newblock Minimum vectors for real reductive algebraic groups.
		\newblock {\em J. London Math. Soc. (2)}, 42(3):409--429, 1990.
		
		\bibitem{MR2931326}
		Adam~S. Sikora.
		\newblock Character varieties.
		\newblock {\em Trans. Amer. Math. Soc.}, 364(10):5173--5208, 2012.
		
		\bibitem{MR169956}
		Andr\'e Weil.
		\newblock Remarks on the cohomology of groups.
		\newblock {\em Ann. of Math. (2)}, 80:149--157, 1964.
			
\end{thebibliography}
	
\begin{small}
\noindent \textsc{Departament de Matem\`atiques,  Universitat Aut\`onoma de Barcelona, and
\\
Centre de Recerca Matem\`atica (UAB-CRM)
08193 Cerdanyola del Vall\`es, Spain }

\noindent \textsf{alejandro.garcia.sanchez@uab.cat}

\noindent \textsf{joan.porti@uab.cat}
\end{small}

\end{document}